\documentclass[11pt,leqno]{article}
\usepackage{amsmath, amscd, amsthm, amssymb, graphics, xypic, mathrsfs, setspace, fancyhdr, times, bm, enumitem}
\DeclareFontFamily{OT1}{pzc}{}
\DeclareFontShape{OT1}{pzc}{m}{it}%
             {<-> s * [1.195] pzcmi7t}{}
\DeclareMathAlphabet{\mathscr}{OT1}{pzc}%
                                 {m}{it}
\usepackage[colorlinks=true,pagebackref=true]{hyperref} 
\hypersetup{backref}

\newcommand{\colim}{\operatorname{colim}}
\newcommand{\Spec}{\operatorname{Spec}}

\newcommand{\isomt}{{\stackrel{{\scriptscriptstyle{\sim}}}{\;\rightarrow\;}}}

\renewcommand{\O}{{\mathcal O}}
\renewcommand{\hom}{\operatorname{Hom}}

\newcommand{\cplx}{{\mathbb C}}

\newcommand{\Z}{{\mathbb Z}}

\newcommand{\A}{{\mathbb A}}
\newcommand{\I}{{\mathrm I}}

\newcommand{\aone}{{\mathbb A}^1}

\newcommand{\ho}[1]{\mathscr{H}({#1})}

\newcommand{\Nis}{\operatorname{Nis}}
\newcommand{\Zar}{\operatorname{Zar}} 

\newcommand{\Sm}{\mathrm{Sm}}

\newcommand{\sPre}{\mathrm{sPre}}
\newcommand{\Ho}{\mathrm{Ho}}

\newcommand{\Gr}{\mathrm{Gr}}

\newcommand{\F}{{\mathscr F}}

\newcommand{\longhookrightarrow}{\ensuremath{\lhook\joinrel\relbar\joinrel\rightarrow}}

\newcommand{\op}{{\mathrm{op}}}
\newcommand{\Map}{\operatorname{Map}}
\newcommand{\Sing}[1]{\operatorname{Sing}^{#1}\!}
\newcommand{\Singaone}{\operatorname{Sing}^{\aone}\!\!}
\newcommand{\diag}{\operatorname{diag}}
\newcommand{\aff}{\mathit{aff}}
\newcommand{\sep}{\mathit{sep}}

\newcommand{\Addresses}{{
  \bigskip
  \footnotesize

  A.~Asok, \textsc{Department of Mathematics, University of Southern California,
    Los Angeles, CA 90089-2532, United States;} \textit{E-mail address:} \url{asok@usc.edu}

  \medskip

  M.~Hoyois, \textsc{Department of Mathematics, Massachusetts Institute of Technology,
    Cambridge, MA 02139-4307, United States;} \textit{E-mail address:} \url{hoyois@mit.edu}

  \medskip

  M.~Wendt, \textsc{Mathematics Institute, Zeeman Building, University of Warwick, Coventry CV4 7AL, UK;} \textit{E-mail address:} \url{M.Wendt@warwick.ac.uk}

}}

\newcounter{intro}
\theoremstyle{plain}
\newtheorem{thm}{Theorem}[subsection]

\newtheorem{lem}[thm]{Lemma}
\newtheorem{cor}[thm]{Corollary}
\newtheorem{prop}[thm]{Proposition}
\newtheorem*{claim*}{Claim}  

\newtheorem*{thm*}{Theorem}
\newtheorem*{problem*}{Problem}

\newtheorem{thmintro}{Theorem}

\theoremstyle{definition}
\newtheorem{defn}[thm]{Definition}

\theoremstyle{remark}
\newtheorem{rem}[thm]{Remark}

\newtheorem{ex}[thm]{Example}

\numberwithin{equation}{section}

\begin{document}
\pagestyle{fancy}
\renewcommand{\sectionmark}[1]{\markright{\thesection\ #1}}
\fancyhead{}
\fancyhead[LO,R]{\bfseries\footnotesize\thepage}
\fancyhead[LE]{\bfseries\footnotesize\rightmark}
\fancyhead[RO]{\bfseries\footnotesize\rightmark}
\chead[]{}
\cfoot[]{}
\setlength{\headheight}{1cm}

\author{Aravind Asok\thanks{Aravind Asok was partially supported by National Science Foundation Award DMS-1254892.} \and Marc Hoyois\thanks{Marc Hoyois was partially supported by National Science Foundation Award DMS-1508096.} \and Matthias Wendt\thanks{Matthias Wendt was partially supported by EPSRC grant EP/M001113/1.}}


\title{{\bf Affine representability results in \\ $\aone$-homotopy theory I: vector bundles}}
\date{}
\maketitle

\begin{abstract}
We establish a general ``affine representability" result in $\aone$-homotopy theory over a general base.  We apply this result to obtain representability results for vector bundles in $\aone$-homotopy theory.  Our results simplify and significantly generalize F. Morel's $\aone$-representability theorem for vector bundles.
\end{abstract}

\begin{footnotesize}
\setcounter{tocdepth}{1}
\tableofcontents
\end{footnotesize}

\section{Introduction}
From its inception in the classical work of H. Cartan, S. Eilenberg and its development in the work of J.-P. Serre and H. Bass (among others), the theory of projective modules over a (commutative) ring has drawn inspiration from the theory of topological vector bundles.  J.F. Adams in his MathSciNet review of Bass's paper \cite{Bass} writes that ``[t]his leads to the following programme: take definitions, constructions and theorems from bundle-theory; express them as particular cases of definitions, constructions and statements about finitely-generated projective modules over a general ring; and finally, try to prove the statements under suitable assumptions."  The purpose of this paper is perhaps best understood from the standpoint of this philosophy: we expose an algebro-geometric analog of the classical Pontryagin--Steenrod representability theorem for vector bundles \cite[19.3]{Steenrod}.

To set the stage, fix a commutative Noetherian base ring $k$ and suppose $X$ is a smooth $k$-scheme.  In this setting, F. Morel and V. Voevodsky \cite{MV} constructed a homotopy category for schemes $\ho{k}$ in which homotopies are, very roughly speaking, parameterized by the affine line $\aone_k$.  Paralleling the notation in topology, set $[X,Y]_{\aone} := \hom_{\ho{k}}(X,Y)$.  Write $\mathscr{V}_r(X)$ for the set of isomorphism classes of rank $r$ algebraic vector bundles on $X$ and $\Gr_r=\colim_{N\to\infty}\Gr_{r,N}$ for the infinite Grassmannian, which is an ind-scheme and can also be viewed as an object of $\ho{k}$.

There is a canonical map $\hom(X,\Gr_r) \to \mathscr{V}_r(X)$ defined by pulling back the tautological bundle on $\Gr_r$.
If $X$ is affine, this map is surjective since one may always choose generating global sections.  On the other hand, there is a canonical map $\hom(X,\Gr_r) \to [X,\Gr_r]_{\aone}$ given by taking $\aone$-homotopy classes.  The goal of this paper is to establish the following result, which is the ``affine representability" result mentioned in the title.

\begin{thmintro}[See Theorem \ref{thm:geometricrepresentation}]
\label{thmintro:main}
Suppose $k$ is a Noetherian commutative ring that is regular over a Dedekind domain with perfect residue fields (e.g., $k$ is any field or $k=\Z$). For any smooth affine $k$-scheme $X$ and any integer $r\geq 0$, there is a bijection $\mathscr{V}_r(X)\cong[X,\Gr_r]_{\aone}$ such that the following triangle commutes:
\[
\xymatrix{
\hom(X,\Gr_r) \ar[dr] \ar@{->>}[d] & \\
\ar[r]^-\cong \mathscr{V}_r(X) & [X,\Gr_r]_{\aone}.
}
\]
\end{thmintro}

Theorem \ref{thmintro:main} is essentially the best possible $\aone$-representability theorem for vector bundles for which one could hope, for reasons we explain below. It will be established as a consequence of a general representability theorem in $\aone$-homotopy theory (see Theorem \ref{thm:representability}).  As with the classical Brown representability theorem, the main idea is to reduce the above result to suitable ``Mayer--Vietoris" and ``homotopy invariance" properties.

The functor $X \mapsto \mathscr{V}_r(X)$ cannot be representable on the category of {\em all} smooth $k$-schemes since ``$\aone$-homotopy invariance" can fail, i.e., the map
\begin{equation}
\label{eqn:bassquillen}
\mathscr{V}_r(X) \longrightarrow \mathscr{V}_r(X \times \aone)
\end{equation}
induced by pullback along the projection $X \times \aone \to X$ can fail to be a bijection.  For example, bijectivity of (\ref{eqn:bassquillen}) can fail (i) if $k$ is not regular, or (ii) when $r \geq 2$ and $X$ is a smooth $k$-scheme that is not affine, even if $k$ is a field (see, e.g., \cite{AsokDoran} for related statements about failure of $\aone$-homotopy invariance).  In contrast, if one restricts attention to smooth affine schemes, the assertion that the map (\ref{eqn:bassquillen}) is a bijection for all $r\geq 2$ is effectively the Bass--Quillen conjecture, which is known to hold in a number of cases \cite{Lindel,Popescu}.

If we restrict attention to smooth affine schemes, then the ``Mayer--Vietoris property" must be reformulated since, e.g., an open subscheme of an affine scheme need not be affine.  Here, Morel had a key insight: a good candidate for the ``Mayer--Vietoris property" for affine schemes should involve squares arising from certain analytic isomorphisms along non-zero divisors; he dubbed the result the ``affine Nisnevich Brown--Gersten property" \cite[Definition A.7]{MField}.  We prefer to call this property ``affine Nisnevich excision" to more clearly indicate the parallels with the classical topological situation.

Theorem \ref{thmintro:main} was first established by Morel \cite[Theorem 8.1(3)]{MField} under the much stricter hypotheses that $k$ is a perfect field and $r \neq 2$.\footnote{This form of the statement rests on a strong form of Gabber's presentation lemma \cite[Lemma 1.15]{MField}, the published proof of which requires $k$ to be infinite.}  Morel also states \cite[p. 199 footnote]{MField} that the result holds in the case $r = 2$, but this case depends on unpublished results of L.-F. Moser \cite{Moser}.  The hypothesis that $k$ be a perfect field arises from Morel's appeal to the structure theory of strongly and strictly $\aone$-invariant sheaves, which he developed in the first six chapters of \cite{MField}.  The hypothesis that $r \neq 2$ arises, in part, from Morel's reliance on results of T. Vorst regarding a $K_1$-analog of the Bass--Quillen conjecture \cite{Vorst1,Vorst2}.

This state of affairs was unsatisfactory: one has the feeling that all hypotheses beyond excision and homotopy invariance are of a purely technical nature and the case $r = 2$ is undoubtedly important. Our goal here is to give a proof of Theorem \ref{thmintro:main} that justifies this feeling.  While affine Nisnevich excision is, not unexpectedly, central to our discussion, our approach is entirely independent of Morel's book.

Our proof relies on two ideas.  First, instead of working directly with the Grassmannian, we work with the classifying stack of vector bundles and only mention the Grassmannian at the very end; we learned this idea from M. Schlichting, who used it to simplify Morel's proof by avoiding appeal to Vorst's $K_1$-analog of the Bass--Quillen conjecture \cite[Theorem 6.22]{Schlichting}.  This approach places the case $r = 2$ on the same footing as $r \geq 3$.  However, Schlichting's proof still uses the results of \cite[Appendix A]{MField} and therefore implicitly on the results of the first six chapters of \cite{MField}; as a consequence, his $\aone$-representability results are restricted to infinite perfect base fields.

The second idea appears in the proof of Theorem \ref{thm:affinedescent}, which shows that affine Nisnevich excision is sufficient for ``Nisnevich local" representability.  Part (i) of Theorem \ref{thm:affinedescent}, which is a Zariski version of the idea just mentioned is, in a sense, ``classical"; the idea goes back to Thomason and was exposed in a special case by Weibel \cite{Weibel}, though we believe our proof gets to the heart of the matter as it makes no mention of homotopy invariance.  We view Part (ii) of Theorem \ref{thm:affinedescent} as the key new ingredient (cf. \cite[p. 231 before Theorem A.11]{MField}).  In turn, the key technical ingredient in the proof of Theorem \ref{thm:affinedescent}(ii) is Lemma \ref{lem:etalenbd}, which is a geometric result about refining \'etale neighborhoods.  This refinement result, which is codified in Proposition \ref{prop:affineNisnevich}, is considerably more subtle than its Zariski counterpart (i.e., Proposition \ref{prop:affineZariski}).  Combining these results with ideas of Voevodsky (see Theorem \ref{thm:cd}) then yields our comparison between affine Nisnevich and Nisnevich excision.

Given Theorems \ref{thm:affinedescent} and \ref{thm:singhasaffinebgproperty}, the proof of Theorem \ref{thmintro:main} is, as one would hope, formally reduced to two classical inputs:  (i) descent for vector bundles, which guarantees that affine Nisnevich excision holds, and (ii) the solution to the Bass--Quillen conjecture, which guarantees that $\aone$-homotopy invariance holds.  Descent for vector bundles holds in great generality, and the restrictions on $k$ in Theorem \ref{thmintro:main} are imposed entirely to allow us to appeal to affirmations of the Bass--Quillen conjecture in the work of H. Lindel and D. Popescu \cite{Lindel,Popescu}.

Theorem \ref{thmintro:main} is not just of abstract significance.  When combined with techniques of obstruction theory, the representability results above have already been used to make significant progress on a number of classical questions about vector bundles on smooth affine varieties.  In brief, \cite{AsokFaselSpheres} establishes a conjecture of M. Nori on unimodular rows, \cite{AsokFaselThreefolds} establishes a cancellation theorem for rank $2$ vector bundles on smooth affine threefolds over algebraically closed fields, \cite{AsokFaselThreefolds} establishes the first non-trivial cases of a splitting conjecture of M.P. Murthy, and \cite{AsokFaselHopkins} describes new obstructions to existence of algebraic structures on topological vector bundles on smooth affine varieties.  Moreover, one can also establish existence results for vector bundles on motivic spheres over $\Spec \Z$ using a clutching construction \cite{AsokDoranFasel}.  We have focused this paper on representability results for vector bundles with the goal of keeping the exposition as simple as possible.

Furthermore, the techniques developed in this paper are also very useful.  In \cite{AffineRepresentabilityII}, we will use Theorem \ref{thmintro:main} (and some extensions) to establish affine representability results for principal $G$-bundles and various homogeneous spaces for linear algebraic groups $G$.  For example, we establish affine representability results for oriented and symplectic vector bundles in the same generality as Theorem~\ref{thmintro:main}, and similar representability results for more general reductive groups when the base is an infinite field.  Not unexpectedly, such generalizations require much more background from the theory of algebraic groups, and we have chosen to expose them separately to avoid encumbering the main arguments here with additional terminological baggage.

\subsubsection*{Overview of sections}
The beginning of each section contains more detailed information about its contents.  Section \ref{s:comparisonoftopologies} recalls Voevodsky's theory of cd-structures, which is a convenient language to study the various Grothendieck topologies that we consider.  The only prerequisites for this section are some knowledge about Grothendieck topologies and some basic algebraic geometry.  Section \ref{s:descent} studies fibrancy in the local homotopy theory of simplicial presheaves and thus requires familiarity with model categories and their localizations and local homotopy theory (\cite{Hirschhorn} is a standard reference for the former, and \cite[Chapters 4-5]{JardineLHT} provide a reference for the latter, though we caution the reader familiar with the local homotopy theory to pay careful attention to our definitions).  Section \ref{s:homotopylocalization} studies a Bousfield localization of the homotopy category of simplicial presheaves with respect to an ``interval"; this abstract setting was studied by Morel and Voevodsky in \cite{MV}, but can be read largely independently of their work.  The results of Sections \ref{s:descent} and \ref{s:homotopylocalization} are written to emphasize their largely formal nature.  Finally, Section \ref{s:representability} puts everything together and requires some familiarity with the definitions of \cite{MV}, but does not use any of the more difficult results of that paper.

\subsubsection*{Acknowledgements}
The authors would like to thank Fabien Morel for helpful discussions about the results of \cite{MField}, Marco Schlichting for helpful discussions about \cite{Schlichting}, and Oliver R\"ondigs for helpful comments on a preliminary draft of this work.  The authors would also like to heartily thank Brad Drew for a number of discussions that clarified some technical points and issues of presentation and Mike Hopkins for useful discussions and encouragement.  Finally, we thank the referees for their careful reading of the paper and a number of suggestions for improvement of the exposition.

\subsubsection*{Preliminaries/Notation}
Throughout this paper, we use the symbol $S$ for a quasi-compact, quasi-separated base scheme. We write $\Sm_S$ for the category of finitely presented smooth $S$-schemes, $\Sm_S^\sep\subset\Sm_S$ for the full subcategory of separated schemes (in the absolute sense), and $\Sm_S^\aff\subset\Sm_S^\sep$ for the full subcategory of affine schemes (again, in the absolute sense).

We use the standard definition of a Grothendieck topology in terms of sieves \cite[Exp.\ II, D\'efinition 1.1]{SGA41}.
If $\mathbf C$ is a category equipped with a Grothendieck topology $t$, we call a family of morphisms $\{U_i\to X\}_{i\in I}$ in $\mathbf C$ a \emph{$t$-cover} if it generates a $t$-covering sieve.

We define the Nisnevich topology on $\Sm_S$, $\Sm_S^\sep$, and $\Sm_S^\aff$ to be generated by finite families of \'etale maps $\{U_i \to X\}$ admitting a splitting sequence by finitely presented closed subschemes, in the sense of \cite[\S3, p. 97]{MV}; this definition is studied, for example, in \cite[\S1]{DAGXI}.  This definition is equivalent to the standard definition \cite[\S 3 Definition 1.2]{MV} if $S$ is Noetherian, by \cite[\S3 Lemma 1.5]{MV}. In fact, one can show that they are equivalent even if $S$ is not Noetherian, but we will not need this fact.

\section{Comparing topologies generated by cd-structures}
\label{s:comparisonoftopologies}
In this section, we recall Voevodsky's definition of a cd-structure on a category and the associated Grothendieck topology \cite[Definition 2.1]{VVcd}.  For the purposes of applications in this paper, it suffices to observe that Voevodsky's theory provides a convenient axiomatization of topologies similar to the Zariski or Nisnevich topology; see Example \ref{ex:cdstructures} for more details.  Our main goal is to establish various comparison results for topologies generated by different cd-structures on a given category.

The main result of this section is Proposition \ref{prop:affineNisnevich}, which provides a particularly simple description of the Nisnevich topology on smooth schemes in terms of a restricted class of ``affine" distinguished squares.  As a warm-up to the main result, we begin by establishing Proposition \ref{prop:affineZariski} which provides an analogous characterization of the Zariski topology in the spirit of our main results.  The key technical step in our proof of Proposition \ref{prop:affineNisnevich} is Lemma \ref{lem:etalenbd}, which provides a refinement result for Nisnevich neighborhoods akin to that used in the proof of Proposition \ref{prop:affineZariski}.  To streamline the proof of Lemma \ref{lem:etalenbd}, we have collected various results about henselization of pairs in Subsection \ref{ss:henselizationofpairs}.

\subsection{cd-structures: definitions and first comparisons}
\label{ss:cdstructures}
\begin{defn}
\label{defn:cdstructure}
Suppose ${\mathbf C}$ is a small category that has an initial object $\emptyset$.
\begin{enumerate}[noitemsep,topsep=1pt]
\item A cd-structure $P$ on ${\mathbf C}$ is a collection of commutative squares in $\mathbf{C}$ such that, if $Q \in P$, and $Q'$ is isomorphic to $Q$, then $Q' \in P$; the squares in $P$ will be called {\em $P$-squares}.
\item If $P$ is a cd-structure on $\mathbf{C}$, then the Grothendieck topology $t_P$ on ${\mathbf C}$ generated by $P$ is the coarsest topology such that:
\begin{itemize}[noitemsep,topsep=1pt]
	\item the empty sieve covers $\emptyset$;
	\item for every square $Q\in P$ of the form
	\begin{equation}
	\label{eqn:square}
	\xymatrix{
	W \ar[r]\ar[d] & V \ar[d]\\
	U \ar[r] & X,
	}
	\end{equation}
	the sieve on $X$ generated by $U\to X$ and $V\to X$ is a $t_P$-covering sieve.
\end{itemize}
\end{enumerate}
\end{defn}

\begin{ex}
	\label{ex:cdstructures}
We will consider the following examples of cd-structures.
\begin{enumerate}[noitemsep,topsep=1pt]
	\item The \emph{Zariski cd-structure} $\Zar$ on $\Sm_S$ consists of Cartesian squares
	\[
	\xymatrix{
	U\times_XV \ar[r] \ar[d] & V \ar[d] \\
	U \ar[r] & X
	}
	\]
	where $U,V\subset X$ are open subschemes such that $U\cup V=X$.
	\item The \emph{Nisnevich cd-structure} $\Nis$ on $\Sm_S$ consists of Cartesian squares
\[
\xymatrix{
U\times_XV \ar[r] \ar[d] & V \ar[d]^{\pi} \\
U \ar[r]_{j} & X
}
\]
where $j$ is an open immersion, $\pi$ is \'etale, and $\pi$ induces an isomorphism $V\times_XZ\cong Z$, where $Z$ is the reduced closed complement of $j$.
\item The \emph{restricted Nisnevich cd-structure} $\mathrm{RNis}$ on $\Sm_S^\sep$ has squares as in the Nisnevich cd-structure, where in addition $\pi$ is affine.
\item The \emph{affine Zariski cd-structure} $\mathrm{AffZar}$ on $\Sm_S^\aff$ consists of Cartesian squares
\[
\xymatrix{
\Spec(A_{fg}) \ar[r] \ar[d] & \Spec(A_g) \ar[d] \\
\Spec(A_f) \ar[r] & \Spec(A)
}
\]
where $f,g\in A$ generate the unit ideal.
\item
The \emph{affine Nisnevich cd-structure} $\mathrm{AffNis}$ on $\Sm_S^\aff$ consists of Cartesian squares
\[
\xymatrix{
\Spec(B_f) \ar[r] \ar[d] & \Spec(B) \ar[d]^{\pi} \\
\Spec(A_f) \ar[r] & \Spec(A)
}
\]
where $f\in A$, $\pi$ is \'etale and induces an isomorphism $A/f\cong B/f$.
\end{enumerate}
\end{ex}

Note that if $\pi: V \to X \leftarrow U : j$ gives rise to a Nisnevich square in $\Sm_S$, then $\pi$ induces an isomorphism $V\times_XZ\to Z$ for \emph{any} closed subscheme $Z\subset X$ complementary to $j$, by the topological invariance of the \'etale site \cite[18.1.2]{EGA44}. Since $U$ is quasi-compact, there always exists such a $Z$ which is finitely presented over $X$ \cite[2.6.1 (c)]{Thomason}.

The topology $t_{\Zar}$ on $\Sm_S$ generated by the Zariski cd-structure is the usual Zariski topology, since schemes in $\Sm_S$ are quasi-compact and quasi-separated.

The proof of \cite[\S 3 Proposition 1.4, p. 96]{MV} shows that the topology $t_{\Nis}$ on $\Sm_S$ generated by the Nisnevich cd-structure is the Nisnevich topology. The same proof works if $\Sm_S$ is replaced by $\Sm_S^\sep$, but it breaks down for $\Sm_S^\aff$, because an open subscheme of an affine scheme need not be affine. Our main result below is that $t_{\mathrm{AffNis}}$ nevertheless coincides with the Nisnevich topology on $\Sm_S^\aff$. Along the way, we will also show that $t_{\mathrm{AffZar}}$ coincides with the Zariski topology on $\Sm_S^\aff$ and that $t_{\mathrm{RNis}}$ coincides with the Nisnevich topology on $\Sm_S^\sep$.

\begin{prop}
	\label{prop:affineZariski}
	The topology $t_{\mathrm{AffZar}}$ on $\Sm_S^{\aff}$ coincides with the Zariski topology.
\end{prop}

\begin{proof}
Set $t:= t_{\mathrm{AffZar}}$. By definition, $t$ is coarser than the Zariski topology.  Therefore, we have to show that every Zariski cover in $\Sm_S^{\aff}$ is a $t$-cover. A general Zariski cover of $X=\Spec(A)$ is refined by one of the form $\{X_{f_1},\dots,X_{f_n}\}$ where $(f_1,\dots,f_n)=A$. We prove that such a cover is a $t$-cover by induction on $n$. For $n\leq 2$, there is nothing to prove. Choose $g\in (f_1,\dotsc,f_{n-1})$ such that $(g,f_n)=A$. Then $\{X_g,X_{f_n}\}$ is a $t$-cover of $X$. By the induction hypothesis, $\{X_{gf_1},\dots,X_{gf_{n-1}}\}$ is a $t$-cover of $X_g$. Hence, $\{X_{gf_{1}},\dots, X_{gf_{n-1}}, X_{f_n}\}$ is a $t$-cover which refines the original cover.
\end{proof}

\begin{prop}
	\label{prop:restrictedNis}
	The topology $t_{\mathrm{RNis}}$ on $\Sm_S^\sep$ coincides with the Nisnevich topology.
\end{prop}

\begin{proof}
Set $t := t_{\mathrm{RNis}}$. By definition, $t$ is coarser than the Nisnevich topology. We first show that $t$ is finer than the Zariski topology. Any Zariski cover of $X\in\Sm_S^\sep$ is refined by an affine open cover $\{U_1,\dots,U_n\}$. We prove that this is a $t$-cover by induction on $n$. If $n\leq 2$, there is nothing to prove. Let $Y=U_1\cup\dotsb\cup U_{n-1}$. Then $\{Y,U_n\}$ is a $t$-cover of $X$, since $X$ is separated. By the induction hypothesis, $\{U_1,\dots,U_{n-1}\}$ is a $t$-cover of $Y$. Thus, $\{U_1,\dots,U_n\}$ is a $t$-cover of $X$, as was to be shown.
	
	We now prove that $t$ is finer than the Nisnevich topology. Since the latter is generated by the Nisnevich cd-structure, it suffices to show that, for every Nisnevich square~\eqref{eqn:square}, $\mathfrak U=\{j,\pi\}$ is a $t$-cover of $X$. Let $\{V_i\}_{i\in I}$ be an affine open cover of $V$, and let $X_i=\pi(V_i)\subset X$. Since $X$ is separated, the induced map $V_i\to X_i$ is affine. Then $\{U,X_i\}_{i\in I}$ is a Zariski cover of $X$, hence a $t$-cover. The restriction of $\mathfrak U$ to $U$ is a cover for any topology, and the restriction of $\mathfrak U$ to $X_i$ is refined by the $t$-cover $\{U\cap X_i\to X_i, V_i\to X_i\}$. Hence, $\mathfrak U$ is a $t$-cover.
\end{proof}

\begin{cor}
	\label{cor:restrictedNis}
	For every restricted Nisnevich square $Q$ as in \eqref{eqn:square} with $X\in\Sm_S^{\aff}$, let $\mathfrak U_Q\subset \Sm_S^\aff/X$ be the Nisnevich cover consisting of $V\to X$ and all affine schemes mapping to $U$. Then the covers $\mathfrak U_Q$, together with the empty cover of $\emptyset$, generate the Nisnevich topology on $\Sm_S^\aff$.
\end{cor}

\begin{proof}
	Let $t$ be the topology generated by such covers. We must show that $t$ is finer than the Nisnevich topology. For every $X\in\Sm_S^\sep$, let $J(X)$ be the collection of sieves $R$ on $X$ such that, for every $f\colon Y\to X$ with $Y\in\Sm_S^\aff$, $f^*(R)\cap \Sm_S^\aff$ is a $t$-covering sieve of $Y$. It is straightforward to verify that $J$ is a Grothendieck topology on $\Sm_S^\sep$ (in fact, it is the finest topology for which the inclusion $(\Sm_S^{\aff},t)\subset (\Sm_S^\sep,J)$ is cocontinuous, cf.\ \cite[Exp.\ III, 3.7]{SGA41}). By definition of $t$, the covering sieves in $J$ contain all sieves generated by restricted Nisnevich squares, as well as the empty sieve on $\emptyset$. Hence, by Proposition~\ref{prop:restrictedNis}, $J$ is finer than the Nisnevich topology. In particular, if $R$ is a Nisnevich covering sieve of $X$ in $\Sm_S^\aff$ and $\hat R$ is the sieve in $\Sm_S^\sep$ generated by $R$, then $\hat R$ is a $J$-covering sieve. By definition of $J$, this implies that $R=\hat R\cap \Sm_S^\aff$ is a $t$-covering sieve of $X$, as desired.
\end{proof}

\subsection{Preliminaries on henselization of pairs}
\label{ss:henselizationofpairs}
Next, we would like to compare the affine Nisnevich cd-structure with the Nisnevich topology on $\Sm_S^{\aff}$.  To do this, we recall some facts about henselization of closed subschemes and henselian pairs. We will only consider the affine case: a \emph{pair} $(X,Z)$ is an affine scheme $X$ together with a closed subscheme $Z$. A morphism of pairs $(Y,W)\to (X,Z)$ is a morphism $Y\to X$ such that the composite $W\subset Y\to X$ factors through $Z$. Recall the definition of a henselian pair from, e.g., \cite[\href{http://stacks.math.columbia.edu/tag/09XD}{Tag 09XD Definition 15.7.1}]{stacks-project}. The following lemma summarizes some properties of henselian pairs.

\begin{lem}
\label{lem:henselizationproperties}
The following statements are true.
\begin{itemize}[noitemsep,topsep=1pt]
\item[(i)] A pair $(X,Z)$ is henselian if and only if, given any affine \'etale map $\pi: V \to X$, any map $Z \to V$ over $X$ extends to a section of $\pi$.
\item[(ii)] If $(X,Z)$ is a henselian pair and if $X' \to X$ is an integral morphism, then the pair $(X', Z \times_X X')$ is also a henselian pair.
\item[(iii)] If $(X,Z)$ and $(Z,W)$ are both henselian pairs, then $(X,W)$ is a henselian pair.
\end{itemize}
\end{lem}

\begin{proof}
The first statement is contained in {\cite[\href{http://stacks.math.columbia.edu/tag/09XD}{Tag 09XD Lemma 15.7.7}]{stacks-project}}. The second is \cite[\href{http://stacks.math.columbia.edu/tag/09XD}{Tag 09XD Lemma 15.7.9}]{stacks-project}.  For the third statement, given $\pi:  V \to X$ an \'etale affine morphism, set $\pi': Z \times_X V \to Z$ and consider the diagram
\[
\xymatrix{
& Z \times_X V \ar[d]^{\pi'}\ar[r] & V\ar[d]^{\pi} \\
W \ar[r]& Z \ar[r] & X.
}
\]
A morphism $W \to V$ over $X$ factors through a morphism $W \to Z \times_X V$ over $Z$.  Because $(Z,W)$ is a henselian pair, the forward implication in (i) implies that the morphism $W \to Z \times_X V$ extends to a section of $\pi'$ and this section of $\pi'$ yields a morphism $Z \to V$ over $X$.  Since $(X,Z)$ is a henselian pair, the forward implication in (i) again yields a section of $\pi$.  Since $V \to X$ was arbitrary, the reverse implication of (i) then allows us to conclude that $(X,W)$ is also a Henselian pair.
\end{proof}

The forgetful functor from the category of henselian pairs to the category of pairs has a right adjoint sending a pair $(X,Z)$ to its henselization $(X^h_Z,Z)$.  The details of the construction will be important below.  Suppose $X$ is an affine scheme and $Z \subset X$ is a closed subscheme.
A \emph{Nisnevich neighborhood} of $Z$ in $X$ is a pair $(V,\pi)$ where $V$ is an affine scheme and $\pi: V \to X$ is an \'etale morphism inducing an isomorphism $V\times_XZ \isomt Z$. Nisnevich neighborhoods of $(X,Z)$ form a category $\mathrm{Neib}(X,Z)$, where a morphism is a commuting triangle.
The category $\mathrm{Neib}(X,Z)$ is cofiltered since it has finite limits, and it is essentially small since affine \'etale morphisms are finitely presented. By the proof of \cite[\href{http://stacks.math.columbia.edu/tag/09XD}{Tag 09XD Lemma 15.7.13}]{stacks-project}, the affine scheme $X^h_Z$ is described explicitly as the limit over the category $\mathrm{Neib}(X,Z)$ of the diagram $(V,\pi)\mapsto V$.   For completeness, we record the following formal consequence of the adjunction just described.

\begin{lem}
\label{lem:universalproperty}
If $(X,Z)$ is a pair, then the henselian pair $(X^h_Z,Z)$ is universal among henselian pairs over $(X,Z)$: any morphism of pairs from a henselian pair $(Y,W)$ to $(X,Z)$ factors uniquely through a morphism of pairs $(Y,W) \to (X^h_Z,Z)$.
\end{lem}

\begin{lem}
\label{lem:iteratedhenselization}
	If $X$ is an affine scheme and $W\subset Z\subset X$ is a chain of closed subschemes, then the canonical map
	\[
	(X^h_Z)^h_W \longrightarrow X^h_W,
	\]
	arising from the universal property of the right-hand side, is an isomorphism.
\end{lem}

\begin{proof}
The pair $((X^h_Z)^h_W,W)$ satisfies the universal property for the henselization of $X$ at $W$.
\end{proof}

\subsection{Refining Nisnevich neighborhoods}
\label{ss:nisnevichrefinement}
We now prove that the cd-structure $t_\mathrm{AffNis}$ on $\Sm_S^\aff$ generates the Nisnevich topology.  To see what the difficulty is, consider a Nisnevich square \eqref{eqn:square} where $X$ and $V$ are affine and $U$ is the union of two affines $U_1$ and $U_2$. Then $V$, $U_1$, and $U_2$ form a Nisnevich cover of $X$ in $\Sm_S^\aff$. However, $\{V,U_1\}$ and $\{V,U_2\}$ need not be Nisnevich squares over $X$, because $V\to X$ is only required to be an isomorphism over the complement of $U$.
The next result shows that $V\to X$ can be refined by a composition of affine \'etale maps $V_2\to V_1\to X$ such that $\{V_1,U_1\}$ is a Nisnevich square over $X$ and $\{V_2,U_2\times_XV_1\}$ is a Nisnevich square over $V_1$. Thus, $V_2$, $U_1$, and $U_2\times_XV_1$ form a $t_\mathrm{AffNis}$-cover of $X$ that refines the given Nisnevich cover.

\begin{lem}
\label{lem:etalenbd}
If $X$ is an affine scheme, $Z_1,Z_2\subset X$ are closed subschemes, and $\pi\colon V\to X$ is an affine \'etale map that is an isomorphism over $Z_1\cap Z_2$, then there exists a commutative square
	\[
	\xymatrix{
	V_2 \ar[r]^{\pi_2} \ar[d] & V_1 \ar[d]^{\pi_1} \\
	V \ar[r]^{\pi} & X,
	}
	\]
	of affine schemes with the following properties:
	\begin{itemize}[noitemsep,topsep=1pt]
		\item[(a)] the morphism $\pi_1$ is \'etale and is an isomorphism over $Z_1$;
		\item[(b)] the morphism $\pi_2$ is \'etale and is an isomorphism over $Z_2\times_XV_1$.
	\end{itemize}
\end{lem}

\begin{proof}
By two applications of Lemma~\ref{lem:iteratedhenselization}, the canonical map
\[
((X^h_{Z_1})^h_{Z_2 \times_X X^h_{Z_1}})^h_{Z_1\cap Z_2} \longrightarrow X^h_{Z_1\cap Z_2}
\]
is an isomorphism. Thus, there is a canonical map
\begin{equation}
	\label{eqn:hensel}
X^h_{Z_1\cap Z_2} \longrightarrow (X^h_{Z_1})^h_{Z_2 \times_X X^h_{Z_1}}
\end{equation}
exhibiting the left-hand side as the henselization of the right-hand side at $Z_1\cap Z_2$.
We claim that the map~\eqref{eqn:hensel} is an isomorphism. By the universal property of henselization (i.e., Lemma \ref{lem:universalproperty}), this claim is equivalent to the claim that the pair
\[
((X^h_{Z_1})^h_{Z_2 \times_X X^h_{Z_1}}, Z_1\cap Z_2)
\]
is henselian. This pair is the composition of the pairs
\[
((X^h_{Z_1})^h_{Z_2 \times_X X^h_{Z_1}}, Z_2 \times_X X^h_{Z_1})\quad\text{and}\quad(Z_2 \times_X X^h_{Z_1},Z_1\cap Z_2).
\]
By Lemma~\ref{lem:henselizationproperties} (iii), it suffices to show that both pairs are henselian. The first is henselian by definition. The second is the pullback of the henselian pair $(X^h_{Z_1},Z_1)$ by a closed immersion, and hence is also henselian by Lemma~\ref{lem:henselizationproperties} (ii).

To complete the proof we need to use the explicit description of the henselization as a limit. The pair $(V,\pi)$ is an object in $\mathrm{Neib}(X,Z_1\cap Z_2)$, so that we have a canonical projection $X^h_{Z_1\cap Z_2}\to V$. On the other hand, since~\eqref{eqn:hensel} is an isomorphism,
\[
X^h_{Z_1\cap Z_2}\cong \lim_{(V',\pi')\in \mathrm{Neib}(X^h_{Z_1}, Z_2 \times_X X^h_{Z_1})} V'.
\]
Since finitely presented $\O(X)$-algebras are compact objects in the category of $\O(X)$-algebras, finitely presented affine $X$-schemes are cocompact in the category of affine $X$-schemes. In particular, there exists a factorization
\[
X^h_{Z_1\cap Z_2} \longrightarrow V' \longrightarrow V
\]
for some $(V',\pi')\in \mathrm{Neib}(X^h_{Z_1}, Z_2 \times_X X^h_{Z_1})$.
Since
\[
X^h_{Z_1}=\lim_{(V_1,\pi_1)\in\mathrm{Neib}(X,Z_1)} V_1,
\]
and since $\pi'\colon V'\to X^h_{Z_1}$ is finitely presented, it follows from \cite[1.8.4.2]{EGA41} that $\pi'$ is the pullback of a finitely presented affine map $\pi_2\colon V_2\to V_1$ for some $(V_1,\pi_1)\in\mathrm{Neib}(X,Z_1)$.

 Refining $V_1$ if necessary, it will automatically be the case that $\pi_2$ is \'etale (by \cite[17.7.8]{EGA44}), that it is an isomorphism over $Z_2\times_XV_1$ (by \cite[8.8.2.4]{EGA43}), and that $\pi_1\circ\pi_2$ factors through $\pi$ (by cocompactness of $V$).
\end{proof}

\begin{prop}
\label{prop:affineNisnevich}
The topology $t_{\mathrm{AffNis}}$ on $\Sm_S^{\aff}$ coincides with the Nisnevich topology.
\end{prop}

\begin{proof}
Set $t := t_{\mathrm{AffNis}}$.  Suppose $X\in\Sm_S^{\aff}$, $Z\subset X$ is a finitely presented closed subscheme, and $\pi\colon V\to X$ is an affine \'etale map that is an isomorphism over $Z$.  Let $\mathfrak U$ be the Nisnevich cover of $X$ in $\Sm_S^{\aff}$ consisting of $\pi$ and of all affine schemes mapping to $X\setminus Z$.  By Corollary~\ref{cor:restrictedNis}, it will suffice to show that $\mathfrak U$ is a $t$-cover. We establish this statement by induction on the minimal number $n$ of generators of the ideal of $Z$.
	
If $n=0$, then $Z=X$ and $\pi$ is an isomorphism and there is nothing to check. If $n=1$, then $X\setminus Z=X_{f}$ for some $f\in\O(X)$, so $\mathfrak U$ is a $t$-cover by definition of the affine Nisnevich cd-structure. Suppose that $n\geq 2$, and choose $f_1,\dotsc,f_n\in \O(X)$ such that $Z=V(f_1,\dotsc,f_n)$. Let $Z_1=V(f_1,\dotsc,f_{n-1})$ and $Z_2=V(f_n)$, so that $Z=Z_1\cap Z_2$. By Lemma~\ref{lem:etalenbd}, $\pi\colon V\to X$ is refined by a composite of affine \'etale maps $\pi_2\colon V_2\to V_1$ and $\pi_1\colon V_1\to X$, where $\pi_1$ is an isomorphism over $Z_1$ and $\pi_2$ is an isomorphism over $Z_2\times_XV_1$. Note that $\{V_2, V_1\setminus(Z_2\times_XV_1)\}$ is a $t$-cover of $V_1$. Applying the induction hypothesis, we see that $\{V_1,X\setminus Z_1\}$ is refined by a $t$-cover of $X$. Thus, $\{V_2,V_1\setminus(Z_2\times_XV_1), X\setminus Z_1\}$ is refined by a $t$-cover of $X$, which clearly refines $\mathfrak U$. Hence, $\mathfrak U$ is a $t$-cover.
\end{proof}

\section{Fibrancy and cd-structures}
\label{s:descent}
\label{sub:descent}
In this section, we study simplicial presheaves on a category equipped with a topology generated by a cd-structure.  Such categories may be equipped with a model structure for which the weak equivalences are defined ``locally" with respect to the topology. The main goals of this section are to establish Theorem \ref{thm:cd} and Theorem \ref{thm:affinedescent}.

Theorem \ref{thm:cd} is a variant of a result of Voevodsky characterizing fibrant objects in this local model structure in terms of a simple ``excision'' condition.  Subsection \ref{ss:descent} recalls some basic definitions from local homotopy theory required to precisely formulate this result; as we hope is clear from the proof, this result holds in great generality.  We refer the reader to Remark \ref{rem:t-local} for a comparison of the local homotopy theory we use to other commonly used versions in the literature.

The other main result of this section is Theorem \ref{thm:affinedescent}.  For this result, we restrict our attention to the local homotopy theory of simplicial presheaves on smooth schemes with respect to the Zariski or Nisnevich topologies.  Building on Theorem \ref{thm:cd} and using the characterizations of the Zariski and Nisnevich topologies on smooth schemes established in Propositions \ref{prop:affineZariski} and \ref{prop:affineNisnevich}, this result essentially provides a convenient characterization of fibrancy in the local model structure in terms of the affine Zariski or Nisnevich squares introduced in Example \ref{ex:cdstructures}.

\subsection{Descent}
\label{ss:descent}
Let $\mathbf C$ be a small category. We denote by $\sPre(\mathbf C)$ the category of simplicial presheaves on $\mathbf C$, which we regard as a model category with the injective model structure: the weak equivalences are the objectwise weak equivalences, the cofibrations are the monomorphisms, and the fibrations are determined by the right lifting property.  We denote by $\Map(\mathcal F,\mathcal G)$ the derived mapping space between two simplicial presheaves $\mathcal F$ and $\mathcal G$, in the sense of \cite[Notation 17.4.2]{Hirschhorn}.

The model category $\sPre(\mathbf C)$ is simplicial, proper, and combinatorial \cite[Proposition A.2.8.2]{HTT}. In particular, we may use the machinery of left Bousfield localizations \cite[\S A.3.7]{HTT}, which we briefly recall. For any set $\mathcal S$ of morphisms in $\sPre(\mathbf C)$, a simplicial presheaf $\mathcal F$ is \emph{$\mathcal S$-local} if, for every $f\colon \mathcal G\to \mathcal H$ in $\mathcal S$, the induced map
\[
f^*\colon \Map(\mathcal H,\mathcal F) \longrightarrow \Map(\mathcal G,\mathcal F)
\]
is a weak equivalence. A morphism $f\colon \mathcal G\to\mathcal H$ is an \emph{$\mathcal S$-equivalence} if, for every $\mathcal S$-local $\mathcal F$, the above map is a weak equivalence. The $\mathcal S$-equivalences are the weak equivalence of a new model structure $\mathcal{S}^{-1}\sPre(\mathbf C)$ on $\sPre(\mathbf C)$, called the \emph{$\mathcal S$-local model structure}. The cofibrations in $\mathcal{S}^{-1}\sPre(\mathbf C)$ are still the monomorphisms, and the fibrant objects are the fibrant objects in $\sPre(\mathbf C)$ that are $\mathcal S$-local. The identity functors form a Quillen adjunction
\[
\xymatrix{
\sPre(\mathbf C) \ar@<.17em>[r] & \ar@<.17em>[l] \mathcal{S}^{-1}\sPre(\mathbf C).
}\]
Moreover, the right derived functor $\Ho(\mathcal{S}^{-1}\sPre(\mathbf C))\to \Ho(\sPre(\mathbf C))$ is fully faithful and its essential image is precisely the subcategory of $\mathcal S$-local objects.

\begin{defn}
\label{dfn:t-local}
Let $\mathbf C$ be a small category and $t$ a Grothendieck topology on $\mathbf C$.  Let $\mathcal S_t$ be the set of covering sieves for $t$, viewed as monomorphisms of simplicial presheaves $R\hookrightarrow X$ with $X$ representable.  We say that a simplicial presheaf on $\mathbf C$ \emph{satisfies $t$-descent} or is \emph{$t$-local} if it is $\mathcal S_t$-local, and we say that a morphism of simplicial presheaves is a \emph{$t$-local weak equivalence} if it is an $\mathcal S_t$-equivalence. We denote by $R_t$ a fibrant replacement functor in the $\mathcal S_t$-local model structure.
\end{defn}

\begin{ex}\label{ex:t-local}
If $\F$ is a presheaf of sets, viewed as a discrete simplicial presheaf, then $\F$ is $t$-local if and only if it is a $t$-sheaf.  Moreover, $R_t\F$ is weakly equivalent to the ordinary sheafification of $\F$.  Indeed, the canonical map from $\F$ to its sheafification is a $t$-local weak equivalence, by \cite[Theorem A.5]{DHI}.  In particular, a morphism between presheaves of sets is a $t$-local weak equivalence if and only if it becomes an isomorphism after sheafification.
\end{ex}

If $\mathfrak U=\{U_i\to X\}_{i\in I}$ is a family of maps in $\mathbf C$, we denote by $\check C(\mathfrak U)$ its \v Cech nerve, which is a simplicial presheaf on $\mathbf C$ with an augmentation to $X$ (see \cite[\S4.1]{DHI}).  The following result gives a useful characterization of $t$-local objects in terms of ``\v Cech descent."

\begin{lem}
\label{lem:descent}
Suppose $\mathbf C$ is a small category. For every $X\in\mathbf C$, let $K(X)$ be a set whose elements are families of maps $\{U_i\to X\}_{i\in I}$ with target $X$. Suppose that, for every $\{U_i\to X\}_{i\in I}$ in $K(X)$ and every map $Y\to X$ in $\mathbf C$, the fiber products $U_i\times_XY$ exist and $\{U_i\times_XY\to Y\}_{i\in I}$ is in $K(Y)$ (up to reindexing). Let $t$ be the coarsest topology on $\mathbf C$ such that the elements of $K(X)$ are $t$-covers for all $X\in\mathbf C$, and let
\[\mathcal S=\{\check C(\mathfrak U)\to X| X\in\mathbf C,\; \mathfrak U\in K(X)\}.\]
A simplicial presheaf $\F$ on $\mathbf C$ is $t$-local if and only if it is $\mathcal S$-local, i.e., if and only if for every $X\in\mathbf C$ and every $\{U_i\to X\}_{i\in I}\in K(X)$, the canonical map
	\[
	\F(X) \longrightarrow \operatorname{holim}_{n\in \Delta} \prod_{i_0,\dots,i_n\in I} \F(U_{i_0}\times_X\dotsb\times_X U_{i_n}).
	\]
	is a weak equivalence.
\end{lem}

\begin{proof}
	For any $\mathfrak U\in K(X)$, the map $\check C(\mathfrak U)\to X$ factors through the sieve $R$ generated by $\mathfrak U$, and the resulting map $\check C(\mathfrak U)\to R$ is a weak equivalence \cite[Proposition A.1]{DHI}. Hence, up to weak equivalence, $\mathcal S$ is a subset of the set of all $t$-covering sieves. By assumption, the sieves in $\mathcal S$ are stable under pullbacks and generate the topology $t$. The result now follows from \cite[Corollary C.2]{Hoyois}.
\end{proof}

\begin{rem}
\label{rem:t-local}
The $t$-local model structure on simplicial presheaves defined above should not be confused with the Jardine model structure defined in \cite{Jardine} (which is Quillen equivalent to the Joyal model structure on simplicial sheaves used in \cite{MV}).  The Jardine model structure is the left Bousfield localization of $\sPre(\mathbf C)$ whose local objects are the simplicial presheaves satisfying \emph{$t$-hyperdescent}, i.e., descent for arbitrary hypercovers. Furthermore, for sites with enough points, the weak equivalences in the Jardine model structure are precisely the stalkwise weak equivalences.

In general, the Jardine model structure is a left Bousfield localization of the $t$-local model structure we consider, and the two model structures may actually differ: see \cite[Example A.9]{DHI} for an explicit example.  However, if $S$ is a Noetherian scheme of finite Krull dimension, $\mathbf{C} = \Sm_S$ and $t$ is the Nisnevich topology, then the two model structures are equivalent: given the characterization of Nisnevich-local objects provided by Theorem~\ref{thm:cd} below, this is the statement of \cite[\S3, Lemma 1.18]{MV}; see also \cite[Example A.10]{DHI} for more discussion of this point.
\end{rem}

\subsection{Excision and descent}
\label{ss:excisionanddescent}
\begin{defn}
\label{defn:excision}
	Let $\mathbf C$ be a small category with an initial object $\emptyset$ and let $P$ be a cd-structure on $\mathbf C$. A simplicial presheaf $\F$ on $\mathbf C$ \emph{satisfies $P$-excision} if
	\begin{itemize}[noitemsep,topsep=1pt]
		\item $\F(\emptyset)$ is weakly contractible;
		\item for every square $Q$ in $P$, $\F(Q)$ is homotopy Cartesian.
	\end{itemize}
\end{defn}

\begin{rem}
Simplicial presheaves that satisfy excision as in Definition \ref{defn:excision} have been given various different names.  In \cite{VVcd}, Voevodsky calls such presheaves \emph{$P$-flasque}.  Morel and Voevodsky \cite[\S 3.1]{MV} and Morel \cite[Appendix A]{MField} use different terminology, which pays homage to the foundational work of Brown and Gersten.
\end{rem}

\begin{defn}[{\cite[Definition 3.4.6]{BorceuxIII}}]
\label{defn:strictlyinitial}
An initial object $\emptyset$ of a category $\mathbf{C}$ is called {\em strict} if any morphism $X \to \emptyset$ is an isomorphism.
\end{defn}

\begin{rem}
Equivalently, the definition can be restated as saying that if $X \to \emptyset$ is any morphism, then $X$ is initial.  The condition that an initial object is strict means that the canonical map $\emptyset \to Y$ is always a monomorphism.  If $\mathbf{C}$ is a Grothendieck topos, then by \cite[Proposition 3.4.7]{BorceuxIII} the initial object is strict.
\end{rem}

The following theorem is a variant of \cite[Corollary 5.10]{VVcd}; note that its hypotheses hold for each cd-structure considered in Example \ref{ex:cdstructures}.

\begin{thm}[Voevodsky]
	\label{thm:cd}
	Let $\mathbf{C}$ be a small category with a strictly initial object and let $P$ be a cd-structure on $\mathbf{C}$ such that:
	\begin{enumerate}[noitemsep,topsep=1pt]
		\item every square in $P$ is Cartesian;
		\item pullbacks of squares in $P$ exist and belong to $P$;
		\item for every square \eqref{eqn:square} in $P$, $U\to X$ is a monomorphism;
		\item for every square \eqref{eqn:square} in $P$, the square
		\[
		\xymatrix{
		W \ar[r] \ar[d]_{\Delta} & V \ar[d]^{\Delta} \\
		W\times_UW \ar[r] & V\times_XV
		}
		\]
		is also in $P$.
	\end{enumerate}
	Then a simplicial presheaf $\F$ on $\mathbf{C}$ satisfies $P$-excision if and only if it satisfies $t_P$-descent.
\end{thm}

\begin{proof}
Suppose $Q$ is a square in $P$ as in \eqref{eqn:square}. Let $\mathfrak U_Q=\{U\to X, V\to X\}$ be the $t_P$-cover defined by $Q$, and let $c_Q\colon \check C(\mathfrak U_Q)\to X$ be its \v Cech nerve.  By (1) we know that $Q$ is Cartesian, so we can write $W = U \times_X V$.
Let $K_Q$ be the homotopy pushout of the diagram
\[
U \longleftarrow U \times_X V \longrightarrow V.
\]
The map from the homotopy pushout to the ordinary pushout induced by the universal property yields a map $k_Q\colon K_Q\to X$.

Let $e$ be the map from the empty presheaf to the presheaf represented by the initial object of $\mathbf{C}$ (this map is never an isomorphism). Consider the sets of maps
\[
\mathcal C=\{ c_Q | Q\in P\}\cup\{e\}\quad\text{and}\quad\mathcal K=\{k_Q| Q\in P\}\cup\{e\}
\]
in $\sPre(\mathbf C)$.  By (2) and the assumption that the initial object of $\mathbf{C}$ is strictly initial, the $t_P$-covers $\mathfrak U_Q$, together with the empty cover of the initial object, satisfy the hypothesis of Lemma~\ref{lem:descent}. It follows that a simplicial presheaf $\F$ satisfies $t_P$-descent if and only if it is $\mathcal C$-local. On the other hand, by definition, $\F$ satisfies $P$-excision if and only if it is $\mathcal K$-local. Thus, to establish the claim, it will suffice to show that every map in $\mathcal K$ is a $\mathcal C$-equivalence, and conversely that every map in $\mathcal C$ is a $\mathcal K$-equivalence.

For $Q$ as above and $Y \in \{U,V,U \times_X V \}$, the cover $\{U\times_XY\to Y, V\times_XY\to Y\}$ is split, and therefore its \v Cech nerve $c_{Q \times_X Y}\colon\check C(\mathfrak U_Q) \times_X Y \to Y$ has an additional degeneracy; in particular, $c_{Q \times_X Y}$ is a weak equivalence.  Form the homotopy Cartesian square of simplicial presheaves:
	\[
	\xymatrix{
	\check C(\mathfrak U_Q)\times_X^h K_Q \ar[r]^-{p_2} \ar[d]_{p_1} & K_Q \ar[d]^{k_Q} \\
	\check C(\mathfrak U_Q) \ar[r]_{c_Q} & X.
	}
	\]
Since homotopy base change preserves homotopy colimits in the category of simplicial sets (see, e.g., \cite[Theorem 1.4]{RezkFib}), we can conclude that the projection $p_2$ is the homotopy pushout of the three maps $c_{Q\times_XY}\colon \check C(\mathfrak U_Q)\times_XY\to Y$, for $Y\in\{U,V,U\times_XV\}$, all of which are weak equivalences.  Thus, $p_2$ is a weak equivalence. By 2-out-of-3, it remains to show that $p_1$ is both a $\mathcal K$-equivalence and a $\mathcal C$-equivalence.
	
	For $i,j\geq 0$ and $i+j\geq 1$, let $Q_{i,j}=Q\times_XU^i\times_XV^j$, where $U^i$ is the $i$-fold fiber product of $U$ with itself over $X$ and similarly for $V^j$. The map $p_1$ is the homotopy colimit over $n\in\Delta^\op$ of the maps $\coprod_{i+j=n+1}k_{Q_{i,j}}$. Since each square $Q_{i,j}$ is in $P$, by (2), this shows that $p_1$ is a $\mathcal K$-equivalence.
	
	Assume for a moment that both legs in $Q$ are monomorphisms. Then one of the legs of $Q_{i,j}$ is always an isomorphism. For such squares it is clear that $k_{Q_{i,j}}$ is a weak equivalence, so that $p_1$ is a weak equivalence. This shows that $k_Q$ is a $\mathcal C$-equivalence for such $Q$.
	
	Let us prove that, in general, $p_1$ is a $\mathcal C$-equivalence. Since $U\to X$ is a monomorphism by (3), one of the legs of $Q_{i,j}$ is an isomorphism as soon as $i\geq 1$, so we only have to show that $k_{Q_{0,j}}$ is a $\mathcal C$-equivalence for $j\geq 1$. The square $Q_{0,j}$ is the lower square in the following diagram:
	\[
	\xymatrix{
	V^{j-1}\times_XW \ar[r] \ar[d] & V^{j-1}\times_XV \ar[d] \\
	V^{j-1}\times_X(W\times_UW) \ar[r] \ar[d] & V^{j-1}\times_X(V\times_XV) \ar[d] \\
	V^{j-1}\times_XW \ar[r] & V^{j-1}\times_XV.
	}
	\]
	
	Let $Q'$ be the top square. Since the composite of the vertical morphisms are the identity maps, $k_{Q_{0,j}}$ is a retraction of a homotopy cobase change of $k_{Q'}$.
	Therefore, it suffices to show that $k_{Q'}$ is a $\mathcal C$-equivalence.
	By (4) and (2), the top square $Q'$ is in $P$, and by (3) its legs are monomorphisms. We have already proved that $k_{Q'}$ is a $\mathcal C$-equivalence for such $Q'$.
\end{proof}

\begin{rem}
The conditions of Theorem \ref{thm:cd} imply by \cite[Lemma 2.5 and 2.11]{VVcd} that the cd-structure $P$ on $\mathbf{C}$ is complete and regular.  Voevodsky also introduces the notion of a bounded cd-structure \cite[Definition 2.22]{VVcd} and shows that, for $P$ a bounded, complete and regular cd-structure, a Jardine weak equivalence between $P$-excisive simplicial presheaves is a weak equivalence \cite[Lemma 3.5]{VVcd}. This means that for such a cd-structure, $P$-excision is equivalent to $t_P$-hyperdescent (cf.\ Remark \ref{rem:t-local}). Theorem \ref{thm:cd} shows that boundedness of $P$ is not relevant for the equivalence of $P$-excision and $t_P$-descent.
\end{rem}

\subsection{A criterion for Descent}
Let us denote by
\[
i\colon \Sm_S^\aff \longhookrightarrow \Sm_S
\]
the inclusion functor. It induces an adjunction
\[
\xymatrix{
i^*: \sPre(\Sm_S) \ar@<.17em>[r] & \ar@<.17em>[l] \sPre(\Sm_S^\aff): i_*.
}
\]
Since $i^*$ preserves weak equivalences and cofibrations, $(i^*,i_*)$ is a Quillen adjunction, and we will denote by $(i^*,\mathbf Ri_*)$ the derived adjunction. The functor $\mathbf{R}i_*$ is given by homotopy right Kan extension along $i$: if ${\mathscr F}\in\sPre(\Sm_S^{\aff})$ and $Y \in \Sm_S$, then
\[
\mathbf{R}i_* {\mathscr F}(Y) = \operatorname{holim}_{X \in \Sm_S^{\aff}/Y} {\mathscr F}(X).
\]

\begin{lem}
\label{lem:radjointfullyfaithful}
The functor $\mathbf{R}i_*$ is fully faithful.
\end{lem}

\begin{proof}
The claim is equivalent to the assertion that the counit map $i^* \mathbf{R} i_* \to \mathrm{id}$ is an isomorphism.  To demonstrate this latter statement, simply observe that if $Y$ is affine, then the category $\Sm_S^{\aff}/Y$ has $Y$ as a final object, and therefore $\operatorname{holim}_{X \in \Sm_S^{\aff}/Y} \F(X)\simeq \F(Y)$.
\end{proof}

\begin{lem}
\label{lem:cocontinuous}
Let $t$ be either the Zariski or the Nisnevich topology.
The functors $i^*$ and $\mathbf Ri_*$ preserve $t$-local simplicial presheaves and they restrict to an equivalence between the full subcategories of $\Ho(\sPre(\Sm_S))$ and $\Ho(\sPre(\Sm_S^\aff))$ spanned by the $t$-local simplicial presheaves.
\end{lem}

\begin{proof}
	That $i^*$ preserves $t$-local presheaves follows at once from the characterization of $t$-local presheaves in terms of \v Cech descent (Lemma~\ref{lem:descent}). Since
\[
\Map(X,\mathbf Ri_*(\F))\simeq \Map(i^*(X),\F),
\]
$\mathbf Ri_*$ preserves $t$-local objects if and only if, for every $t$-covering sieve $R\hookrightarrow X$ on $\Sm_S$, the restriction $i^*(R)\hookrightarrow i^*(X)$ is a $t$-local weak equivalence. Recall from Example~\ref{ex:t-local} that $t$-local weak equivalences between discrete simplicial presheaves are simply the $t$-local isomorphisms.  That $i^*$ sends $t$-covering sieves to $t$-local isomorphisms is in turn equivalent to $i_*$ preserving $t$-sheaves of \emph{sets}; this follows from \cite[Exp. III, 2.2]{SGA41} together with the fact that $i$ is cocontinuous for either topology.
	
	Thus, the derived adjunction $i^*\dashv \mathbf Ri_*$ restricts to an adjunction between the subcategories of $t$-local presheaves, where the right adjoint is fully faithful (by Lemma \ref{lem:radjointfullyfaithful}). To conclude that it is an equivalence, it remains to show that the left adjoint is conservative.  Indeed, given an adjunction $(F,G)$ with $G$ fully faithful and $F$ conservative, it follows from the commutativity of the triangle
\[
\xymatrix{
F \ar[r]\ar[dr]_{\mathrm{id}} & F\circ G\circ F \ar[d]\\
& F
}
\]
that $G$ is a quasi-inverse to $F$.

Let $f$ be a morphism between $t$-local simplicial presheaves on $\Sm_S$ such that $i^*(f)$ is a weak equivalence. Since any separated scheme admits a $t$-cover whose \v Cech nerve consists of affine schemes, it follows that $f$ is a weak equivalence on $\Sm_S^\sep$. Similarly, since any scheme admits a $t$-cover whose \v Cech nerve consists of separated schemes, we deduce that $f$ is a weak equivalence on all of $\Sm_S$.
\end{proof}

\begin{rem}
	In \cite[\S{}A.2]{MField}, the simplicial presheaf $\mathbf Ri_*i^*(\F)$ is called the \emph{affine replacement} of $\F$.
\end{rem}

The first part of the following theorem is a strengthening of \cite[Lemma A.10]{MField}; Schlichting gives a different proof of this result in \cite[\href{http://arxiv.org/pdf/1502.05424v2.pdf}{Theorem A.2}]{Schlichting}.  The second part of the next result is a significant generalization of \cite[Theorem A.14]{MField}; in particular, it provides an answer to the question posed in \cite[Remark A.15]{MField}.

\begin{thm}
\label{thm:affinedescent}
Let $\F$ be a simplicial presheaf on $\Sm_S$.
\begin{enumerate}[noitemsep,topsep=1pt]
	\item[(i)] If $\F$ satisfies affine Zariski excision, then, for every $X\in\Sm_S^\aff$, the map
	\[
	\F(X) \longrightarrow R_{\Zar}\F(X)
	\]
	is a weak equivalence.
	\item[(ii)] If $\F$ satisfies affine Nisnevich excision, then $R_{\Zar}\F$ is Nisnevich-local.
\end{enumerate}
\end{thm}

\begin{proof}
(i) We consider the commutative square
\[
	\xymatrix{
	i^*(\F) \ar[r]\ar[d] & i^*(R_{\Zar}\F) \ar[d]\\
	R_{\Zar}i^*(\F) \ar[r] & R_{\Zar}i^*(R_{\Zar}\F)
	}
\]
in $\sPre(\Sm_S^\aff)$.
By Proposition~\ref{prop:affineZariski} and Theorem~\ref{thm:cd}, $i^*(\F)$ is Zariski-local, so the left vertical map is a weak equivalence.

By Lemma~\ref{lem:cocontinuous}, $i^*(R_{\Zar}\F)$ is Zariski-local, so the right vertical map is a weak equivalence.
Since $\mathbf{R}i_*$ preserves Zariski-local objects (Lemma~\ref{lem:cocontinuous}), $i^*$ preserves Zariski-local weak equivalences, so the bottom horizontal map is a weak equivalence.
Hence, the top horizontal map is a weak equivalence, as desired.

(ii) By Proposition~\ref{prop:affineNisnevich} and Theorem~\ref{thm:cd}, $i^*(\F)$ is Nisnevich-local.
 By part (i), we deduce that $i^*(R_{\Zar}\F)$ is Nisnevich-local. By Lemma~\ref{lem:cocontinuous}, $R_{\Zar}\F\simeq \mathbf Ri_*i^*(R_{\Zar}\F)$ and $\mathbf Ri_*$ preserves Nisnevich-local objects. Hence, $R_{\Zar}\F$ is Nisnevich-local.
\end{proof}

\section{Sites with interval and homotopy localization}
\label{s:homotopylocalization}
The main goal of this section is to establish Theorem~\ref{thm:singhasaffinebgproperty}, which is an abstract form of a result due originally to Schlichting.  Morel and Voevodsky initially constructed categories like the $\aone$-homotopy category in the more general context of ``sites with interval" and we work roughly in this setting.  We use a variant of the notion of an interval object that is slightly more restrictive than that considered by Morel--Voevodsky; see Section~\ref{ss:intervalsandsing} for more details.  The homotopy category of a site with interval is obtained as a Bousfield localization of the local model structure on simplicial presheaves on a site (with interval) which forces the interval object to be contractible.  Essentially, Theorem~\ref{thm:singhasaffinebgproperty} provides a convenient framework for producing fibrant objects in this further localized model structure, though it is not phrased in this way.  In fact, the proof of Theorem~\ref{thm:singhasaffinebgproperty} can be read independently of the main results of Section \ref{s:descent}.  Furthermore, Corollary~\ref{cor:SingAffineBG} places the general result in the form in which it will be used in the remainder of the paper.

\subsection{Interval objects and the singular construction}
\label{ss:intervalsandsing}
\begin{defn}
\label{defn:intervalobject}
Let ${\mathbf C}$ be a small category.  A {\em representable interval object} in $\mathbf C$ is a quadruple $(\I,m,i_0,i_1)$ consisting of a presheaf $\I$ on $\mathbf C$, a morphism $m: \I \times \I \to \I$ (``multiplication"), and morphisms $i_0,i_1: \ast \to \I$ (``endpoints") such that the following statements are true:
\begin{enumerate}[noitemsep,topsep=1pt]
\item for every $X\in\mathbf C$, the presheaf $X\times\I$ is representable;
\item if $p: \I \to \ast$ is the canonical morphism, then there are equalities of the form:
\begin{itemize}[noitemsep,topsep=1pt]
\item $m(i_0 \times id) = m(id \times i_0) = i_0p$, \text{ and }
\item $m(i_1 \times id) = m(id \times i_1) = id$;
\end{itemize}
\item the morphism $i_0 \coprod i_1: \ast \coprod \ast \to \I$ is a monomorphism.
\end{enumerate}
\end{defn}

\begin{rem}
The notion above is a special case of the notion of {\em interval object} studied in \cite[\S2.3 p. 85]{MV}, which omits the first condition.  The third condition is not needed in any of the arguments below and is included only to conform to the definition in \emph{loc.\ cit.}
\end{rem}

\begin{ex}
\label{ex:smoothschemes}
If $\mathbf{C} = \Sm_S$ or $\Sm_S^{\aff}$, then $\I = \aone_S$ with its usual monoid structure is a representable interval object.  Indeed, there is nothing to check in the former case.  In the latter case, since $\aone_S \to S$ is an affine morphism, if $X \in \Sm_S^{\aff}$, then the fiber product $X \times_S \aone_S$ is affine over $X$ and thus affine.
\end{ex}

\begin{defn}\label{dfn:I-invariant}
Let $\mathbf C$ be a small category and $\I$ a representable interval object in $\mathbf C$. Let $\mathcal S_\I$ be the set of all projections $X\times\I\to X$ with $X\in\mathbf C$.  We say that a simplicial presheaf on $\mathbf{C}$ is \emph{$\I$-invariant} if it is $\mathcal S_\I$-local. In other words, $\F\in\sPre(\mathbf C)$ is $\I$-invariant if, for every $X\in {\mathbf C}$, the projection $X\times \I \to X$ induces a weak equivalence $\F(X)\simeq \F(X\times \I)$.
\end{defn}

Using the two endpoints $i_0,i_1\colon \ast \to \I$, one defines a cosimplicial presheaf $\I^\bullet$ on $C$ with $\I^n= \I^{\times n}$ (see \cite[\S2.3.2 p. 88]{MV} where the notation $\Delta^{\bullet}_{\I}$ is used for this cosimplicial object).  We define the functor $\Sing{\I}\colon \sPre(\mathbf C)\to\sPre(\mathbf C)$ by the formula:
\[
\Sing{\I}\F(X) = \diag \F(X\times \I^\bullet).
\]
There is an obvious natural map $\F\to \Sing{\I}\F$, and for every $\F\in\sPre(\mathbf C)$:
\begin{itemize}[noitemsep,topsep=1pt]
	\item $\Sing{\I}\F$ is $\I$-invariant (see \cite[\S 2 Corollary 3.5]{MV});
	\item the map $\F\to \Sing{\I}\F$ is an $\mathcal S_\I$-equivalence (see \cite[\S 2 Corollary 3.8]{MV}).
\end{itemize}
These properties imply that $\Sing{\I}$ preserves weak equivalences and that the induced endofunctor of the homotopy category $\Ho(\sPre(\mathbf C))$ is left adjoint to the inclusion of the subcategory of $\I$-invariant simplicial presheaves.

\subsection{Excision and the singular construction}
\begin{lem}
\label{lem:diag}
Assume given a commutative diagram of bisimplicial sets of the form
\[
\xymatrix{
Z_*^{12} \ar[r] \ar[d] & Z_*^1 \ar[d]\\
Z_*^2 \ar[r] & Z_*^0
}
\]
that is degreewise homotopy Cartesian. If $\pi_0(Z_*^{0})$ and $\pi_0(Z_*^{1})$ are constant simplicial sets, then the induced diagram
\[
\xymatrix{
\diag(Z_*^{12}) \ar[r] \ar[d] & \diag(Z_*^1) \ar[d]\\
\diag(Z_*^2) \ar[r] & \diag(Z_*^0)
}
\]
is homotopy Cartesian.
\end{lem}

\begin{proof}
This is a consequence of the Bousfield--Friedlander theorem \cite[\S B.4]{BousfieldFriedlander}.  By assumption, the map $\pi_0(Z_*^1)\to \pi_0(Z_*^0)$ is a map of constant simplicial sets and thus a fibration.  It remains to show that a bisimplicial set $Z_*$ such that $\pi_0(Z_*)$ is constant satisfies the $\pi_*$-Kan condition.  Such a bisimplicial set can be written as a coproduct, indexed by the set $\pi_0(Z_0)$, of bisimplicial sets $W_*$ with $\pi_0(W_m)=*$ for all $m$.  It is proved in \cite[\S B.3]{BousfieldFriedlander} that such a $W_*$ satisfies the $\pi_*$-Kan condition. To conclude, one observes that the $\pi_*$-Kan condition is preserved under coproducts.
\end{proof}

\begin{rem}
\label{rem:rezk}
In \cite[Proposition 5.4]{Rezk}, it is shown that the conclusion of Lemma~\ref{lem:diag} holds assuming only that $\pi_0(Z_*^0)$ is constant. This also follows from \cite[Lemma 5.5.6.17]{HA} applied to the $\infty$-topos of presheaves on $\pi_0(Z_0^0)$.
\end{rem}

The following result, which abstracts some ideas of \cite[\href{http://arxiv.org/pdf/1502.05424v2.pdf}{Proof of Theorem 6.22}]{Schlichting}, gives a general procedure to produce spaces having the affine Nisnevich excision property.

\begin{thm}
\label{thm:singhasaffinebgproperty}
Let $\mathbf{C}$ be a small category with a strictly initial object, $\I$ a representable interval object in $\mathbf C$, and $P$ a cd-structure on $\mathbf{C}$ such that
\[
Q\in P\Rightarrow Q\times \I\in P.
\]
If $\F$ is a simplicial presheaf on $\mathbf{C}$ satisfying $P$-excision such that $\pi_0(\F)$ is $\I$-invariant, then $\Sing{\I}\F$ also satisfies $P$-excision.
\end{thm}

\begin{proof}
	Let $\emptyset\in \mathbf{C}$ be the initial object. Since $\emptyset\times \I^n$ maps to $\emptyset$, it is initial. It follows that $\Sing{\I}\F(\emptyset)$ is weakly contractible. Let $Q$ be a square
	\[
	\xymatrix{
	W \ar[r] \ar[d] & V \ar[d] \\
	U \ar[r] & X
	}
	\]
	in $P$. The square $\Sing{\I}\F(Q)$ is the diagonal of the square of bisimplicial sets
	\[
	\xymatrix{
	\F(X\times \I^\bullet) \ar[r] \ar[d] & \F(U\times \I^\bullet) \ar[d] \\
	\F(V\times \I^\bullet) \ar[r] & \F(W\times \I^\bullet)
	}
	\]
	which is degreewise homotopy Cartesian by the assumption on $P$. Moreover, since $\pi_0(\F)$ is $\I$-invariant, the $\pi_0$ of each entry is a constant simplicial set. By Lemma~\ref{lem:diag}, we deduce that $\Sing{\I}\F(Q)$ is homotopy Cartesian.
\end{proof}

\begin{cor}
\label{cor:SingAffineBG}
Let $\F$ be a simplicial presheaf on $\Sm_S^{\aff}$. Suppose that $\F$ satisfies affine Zariski (resp.\ affine Nisnevich) excision and that $\pi_0(\F)$ is $\A^1$-invariant. Then $\Singaone\F$ satisfies affine Zariski (resp.\ affine Nisnevich) excision.
\end{cor}

\begin{proof}
Via Example~\ref{ex:smoothschemes}, this result follows immediately from Theorem~\ref{thm:singhasaffinebgproperty}.
\end{proof}

\section{$\A^1$-representability results for vector bundles}
\label{s:representability}
The goal of this section is to establish our general representability result in the $\A^1$-homotopy category and then to specialize this result to obtain the geometric representation of vector bundles on smooth affine schemes discussed in Theorem~\ref{thmintro:main}.  We begin by establishing Theorem~\ref{thm:representability}, which contains a general ``affine representability result".  This result relies on the main results of the previous sections, namely Theorem~\ref{thm:affinedescent} and Corollary~\ref{cor:SingAffineBG}.  Finally, we establish Theorem \ref{thm:geometricrepresentation}, which, building on some commutative algebra results related to the Bass--Quillen conjecture (see Theorem \ref{thm:lindelpopescu}), is deduced by suitably specializing Theorem~\ref{thm:representability} to classifying spaces for vector bundles.

\subsection{A general representability result}
We now apply our results in the context of $\aone$-homotopy theory.  The $\aone$-homotopy category $\ho{S}$ is constructed as the homotopy category of a left Bousfield localization of $\sPre(\Sm_S)$. Recall from Definitions \ref{dfn:t-local} and \ref{dfn:I-invariant} that $\mathcal S_{\Nis}$ is the set of Nisnevich covering sieves $R\hookrightarrow X$ and that $\mathcal S_{\aone}$ is the set of projections $X\times\A^1\to X$. We define
\[
\ho{S} := \Ho(({\mathcal S_{\Nis}\cup \mathcal S_{\aone}})^{-1}\sPre(\Sm_S)).
\]
Using the adjunction described at the beginning of Section \ref{ss:descent}, we may identify $\ho{S}$ with the full subcategory of $\Ho(\sPre(\Sm_S))$ consisting of Nisnevich-local and $\aone$-invariant simplicial presheaves. If $\mathscr X,\mathscr Y\in\sPre(\Sm_S)$, we set
\[
[\mathscr X,\mathscr Y]_{\A^1} := \hom_{\ho{S}}(\mathscr X,\mathscr Y).
\]
Recall that this can be computed as follows: if $\mathscr Y\to \mathscr Y'$ is an $({\mathcal S_{\Nis}\cup \mathcal S_{\aone}})$-equivalence where $\mathscr Y'$ is Nisnevich-local and $\aone$-invariant, then
\[
[\mathscr X,\mathscr Y]_{\A^1} \cong \hom_{\Ho(\sPre(\Sm_S))}(\mathscr X,\mathscr Y').
\]
In particular, if $X\in\Sm_S$, then $[X,\mathscr Y]_{\A^1}\cong\pi_0\mathscr Y'(X)$.

\begin{rem}
Note that our definition of $\ho{S}$ differs slightly from that studied in \cite[\S3,2.1]{MV}.  Indeed, Morel and Voevodsky define the $\aone$-homotopy category as the homotopy category of a left Bousfield localization of the Joyal model structure on simplicial Nisnevich sheaves on $\Sm_S$. However, they assume from the outset that the base scheme $S$ is Noetherian and has finite Krull dimension.  Under these assumptions, their category is equivalent to ours (see Remark \ref{rem:t-local}).
\end{rem}

\begin{lem}
\label{lem:affineA1invariance}
Let $\F$ be a Zariski-local simplicial presheaf on $\Sm_S$. If $\F$ is $\A^1$-invariant on affines (i.e., for any $X \in \Sm_S^\aff$, the map $\F(X) \to \F(X \times \aone)$ is a weak equivalence), then $\F$ is $\A^1$-invariant.
\end{lem}

\begin{proof}
	We must show that the map $\F(-) \to \F(-\times\A^1)$ is a weak equivalence. This is a map between Zariski-local simplicial presheaves that is a weak equivalence on affines, hence it is a weak equivalence by Lemma~\ref{lem:cocontinuous}.
\end{proof}

\begin{thm}
\label{thm:representability}
Let $\F$ a simplicial presheaf on $\Sm_S$. Suppose that
\begin{enumerate}[noitemsep,topsep=1pt]
	\item[(i)] $\F$ satisfies affine Nisnevich excision;
	\item[(ii)] $\pi_0(\F)$ is $\A^1$-invariant on affine schemes.
\end{enumerate}
Then $R_{\Zar}\Singaone\F$ is Nisnevich-local and $\A^1$-invariant. Moreover, for every $U\in\Sm_S^\aff$, the canonical map
\[
\pi_0\F(U) \longrightarrow [U,\F]_{\A^1}
\]
is an isomorphism.
\end{thm}

\begin{proof}
By Corollary~\ref{cor:SingAffineBG}, $\Singaone\F$ satisfies affine Nisnevich excision.  Therefore, applying Theorem~\ref{thm:affinedescent} (ii),
we conclude that $R_{\Zar}\Singaone\F$ is Nisnevich-local. On the other hand, Theorem~\ref{thm:affinedescent} (i) shows that the map
\begin{equation}\label{eqn:representability}
	\Singaone\F(U) \to R_{\Zar} \Singaone\F(U)
\end{equation}
is a weak equivalence for every affine $U$.  Since $\Singaone\F$ is $\aone$-invariant by definition, it follows immediately that $R_{\Zar}\Singaone\F$ is $\A^1$-invariant on affines.  Applying Lemma~\ref{lem:affineA1invariance} we conclude that $R_{\Zar}\Singaone\F$ is also $\A^1$-invariant.  Now, taking the connected components of \eqref{eqn:representability}, we see that for $U$ affine,
\[
\pi_0\Singaone\F(U)\cong [U,\F]_{\A^1}.
\]
Moreover, condition (ii) of the statement implies that
\[
\pi_0\F(U)\cong \pi_0\Singaone\F(U)
\]
for affine $U$, and this is precisely what we wanted to prove.
\end{proof}

\begin{rem}
	Condition (ii) of Theorem~\ref{thm:representability} cannot be dropped: even if $\F$ is a Nisnevich-local simplicial presheaf, $R_{\Nis}\Singaone\F$ need not be $\A^1$-invariant. One example of this phenomenon is given in \cite[\S 3 Example 2.7]{MV}.
	For an example where $\F$ is a smooth projective scheme over $\cplx$, see \cite[\href{http://arxiv.org/pdf/1312.6388.pdf}{\S 4.1}]{BalweHogadiSawant}.
\end{rem}

\subsection{Homotopy invariance and geometric representability}
Suppose $r \geq 0$ is an integer.  If $X$ is a scheme, write $\mathscr{V}_r(X)$ for the set of isomorphism classes of rank $r$ vector bundles on $X$.  Let $A$ be a commutative ring.  Say that the Bass--Quillen conjecture holds for $A$ in rank $r$ if, for every integer $n\geq 0$, the map
\[
\mathscr{V}_r(\Spec A) \longrightarrow \mathscr{V}_r(\Spec A[t_1,\ldots,t_n])
\]
is a bijection. We recall the following result regarding the Bass--Quillen conjecture.

\begin{thm}[Lindel, Popescu]
\label{thm:lindelpopescu}
Let $A$ be a ring such that, for every maximal ideal $\mathfrak{m}\subset A$, $A_{\mathfrak{m}}$ is ind-smooth over a Dedekind ring with perfect residue fields (for example, $A_{\mathfrak{m}}$ is Noetherian and regular over such a Dedekind ring). Then the Bass--Quillen conjecture holds for $A$ in every rank $r \geq 0$.
\end{thm}

\begin{proof}
The fact that every regular ring map (see, e.g., \cite[\href{http://stacks.math.columbia.edu/tag/07BY}{Tag 07BY Definition 15.32.1}]{stacks-project} for a definition) between Noetherian rings is ind-smooth is Popescu's desingularization theorem \cite[\href{http://stacks.math.columbia.edu/tag/07GC}{Tag 07GC Theorem 16.12.1}]{stacks-project}.
By Quillen's patching theorem \cite[Theorem V.1.6]{Lam}, we may assume that $A$ is ind-smooth over a Dedekind ring $R$ with perfect residue fields. Since any finitely generated projective module over a filtered colimit of rings is extended from some ring in the diagram, we may further assume that $A$ is smooth over $R$. In that case, we can apply \cite[Theorem 4.1]{Popescu} (see also \cite[Theorem 2.2]{SwanPopescu} for clarifications).
\end{proof}

\begin{ex}
The Bass--Quillen conjecture holds for any regular Noetherian ring $A$ such that for every prime $p \in \Z$, $A/pA$ is regular.  Indeed, such a ring is a direct product of rings that are regular over $\Z$ or over $\mathbb{F}_p$ for some prime $p$.
\end{ex}

For $U\in\Sm_S$, write $\mathbf{Vect}_r(U)$ for the groupoid of rank $r$ vector bundles over $U$.   The groupoids $\mathbf{Vect}_r(U)$ are the fibers of a category fibered in groupoids over $\Sm_S$ that we denote by $\mathbf{Vect}_r$. These fibers are only functorial in $U$ in a weak sense, but this is easily rectified by considering the equivalent groupoids $\Gamma(\Sm_S/U,\mathbf{Vect}_r)$, which are strictly functorial in $U$;  this trick is due to Grayson and we refer the reader to \cite[p.~169-170]{Grayson} or \cite[C.4]{FriedlanderSuslin} for more details. We write $B\mathbf{Vect}_r$ for the simplicial presheaf on $\Sm_S$ whose value on $U$ is the nerve of the latter groupoid. Thus, by construction, $\pi_0 B\mathbf{Vect}_r$ is isomorphic to the presheaf $U\mapsto \mathscr V_r(U)$.

Let $\Gr_r$ be the infinite Grassmannian $\colim_N \Gr_{r,N}$, considered as a colimit of representable presheaves on $\Sm_S$. The tautological vector bundles on the Grassmannians $\Gr_{r,N}$ induce a morphism
\[
\gamma_r\colon \Gr_r \to B\mathbf{Vect}_r
\]
in $\sPre(\Sm_S)$. By \cite[\S4, Proposition 2.6]{MV}, $\gamma_r$ is an isomorphism in $\ho{S}$.\footnote{In \cite{MV}, it is assumed that $S$ is Noetherian and has finite Krull dimension, but this is not used in the proof.}
For any $X\in\Sm_S$, composition with $\gamma_r$ gives a map $\hom(X,\Gr_r)\to \mathscr V_r(X)$, which is surjective if $X$ is affine.

\begin{thm}[Affine representability for vector bundles]
\label{thm:geometricrepresentation}
Let $r\geq 0$ be an integer. Suppose that, for every $X\in\Sm_S^\aff$, the Bass--Quillen conjecture holds for $\O(X)$ in rank $r$. Then, for every $X\in\Sm_S^\aff$, there is a bijection
\[
\mathscr{V}_r(X) \cong [X,\Gr_r]_{\aone}
\]
such that the following triangle commutes:
\[
\xymatrix{
\hom(X,\Gr_r) \ar[dr] \ar@{->>}[d] & \\
\ar[r]^-\cong \mathscr{V}_r(X) & [X,\Gr_r]_{\aone}.
}
\]
In particular, this holds for all $r$ if $S$ is ind-smooth over a Dedekind ring with perfect residue fields.
\end{thm}

\begin{rem}
\label{rem:requals1}
Theorem \ref{thm:geometricrepresentation} is trivially true when $r = 0$.  When $r = 1$, the bijection of Theorem \ref{thm:geometricrepresentation} holds for $S$ regular and $X$ an arbitrary smooth $S$-scheme, by \cite[\S 4 Proposition 3.8]{MV}.
\end{rem}

\begin{proof}
Consider the commutative square
\begin{equation}\label{eqn:Gr}
	\xymatrix{
	\hom(X,\Gr_r) \ar[r] \ar@{->>}[d] & [X,\Gr_r]_{\aone} \ar[d] \\
	\ar[r] \mathscr{V}_r(X) & [X,B\mathbf{Vect}_r]_{\aone},
	}
\end{equation}
where $X\in\Sm_S^\aff$ and the vertical maps are induced by $\gamma_r$. Since $\gamma_r$ is an isomorphism in $\ho{S}$, the right vertical map is a bijection. It remains to show that the bottom horizontal map is a bijection.

The classical theory of faithfully flat descent for vector bundles shows that $\mathbf{Vect}_r$ is a stack for the \'etale topology on $\Sm_S$. By \cite[Theorem 1.1]{Hollander}, this is equivalent to $B\mathbf{Vect}_r$ being \'etale-local, in the sense of Definition~\ref{dfn:t-local}. In particular, by Theorem~\ref{thm:cd}, $B\mathbf{Vect}_r$ satisfies Nisnevich excision. The hypothesis regarding the Bass--Quillen conjecture guarantees that $\pi_0B\mathbf{Vect}_r$ is $\aone$-invariant on affine schemes.  Therefore, the simplicial presheaf $B\mathbf{Vect}_r$ satisfies the hypotheses of Theorem \ref{thm:representability} and we conclude that the bottom map in~\eqref{eqn:Gr} is a bijection, as desired.  The final statement follows immediately from Theorem \ref{thm:lindelpopescu}.
\end{proof}

\begin{footnotesize}
\bibliographystyle{alpha}
\bibliography{bundles}

\begin{thebibliography}{{Rez}98}

\bibitem[AD08]{AsokDoran}
A.~Asok and B.~Doran.
\newblock Vector bundles on contractible smooth schemes.
\newblock {\em Duke Math. J.}, 143(3):513--530, 2008.

\bibitem[ADF16]{AsokDoranFasel}
A.~Asok, B.~Doran, and J.~Fasel.
\newblock Smooth models of motivic spheres and their vector bundles.
\newblock {\em Int. Math. Res. Not. IMRN}, 2016.
\newblock Online First \url{http://dx.doi.org/10.1093/imrn/rnw065}.

\bibitem[AF14a]{AsokFaselSpheres}
A.~Asok and J.~Fasel.
\newblock Algebraic vector bundles on spheres.
\newblock {\em J. Topol.}, 7(3):894--926, 2014.

\bibitem[AF14b]{AsokFaselThreefolds}
A.~Asok and J.~Fasel.
\newblock A cohomological classification of vector bundles on smooth affine
  threefolds.
\newblock {\em Duke Math. J.}, 163(14):2561--2601, 2014.

\bibitem[AFH15]{AsokFaselHopkins}
A.~Asok, J.~Fasel, and M.J. Hopkins.
\newblock Obstructions to algebraizing topological vector bundles.
\newblock {\em Preprint}, available at \url{http://arxiv.org/abs/1506.08111},
  2015.

\bibitem[AGV72]{SGA41}
M.~Artin, A.~Grothendieck, and J.-L. Verdier.
\newblock {\em Th{\'e}orie des topos et cohomologie {\'e}tale des sch{\'e}mas},
  volume 269 of {\em Lecture Notes in Mathematics}.
\newblock Springer, 1972.

\bibitem[AHW15]{AffineRepresentabilityII}
A.~Asok, M.~Hoyois, and M.~Wendt.
\newblock Affine representability results in {${\mathbb A}^1$}-homotopy theory
  {II}: principal bundles and homogeneous spaces.
\newblock 2015.
\newblock {\em Preprint} available at
  \url{http://arxiv.org/abs/math/1507.08020}.

\bibitem[Bas64]{Bass}
H.~Bass.
\newblock {$K$}-theory and stable algebra.
\newblock {\em Inst. Hautes \'Etudes Sci. Publ. Math.}, (22):5--60, 1964.

\bibitem[BF78]{BousfieldFriedlander}
A.~K. Bousfield and E.~M. Friedlander.
\newblock Homotopy theory of {$\Gamma $}-spaces, spectra, and bisimplicial
  sets.
\newblock In {\em Geometric applications of homotopy theory ({P}roc. {C}onf.,
  {E}vanston, {I}ll., 1977), {II}}, volume 658 of {\em Lecture Notes in Math.},
  pages 80--130. Springer, Berlin, 1978.

\bibitem[BHS15]{BalweHogadiSawant}
C.~Balwe, A.~Hogadi, and A.~Sawant.
\newblock {$\mathbb{A}^1$}-connected components of schemes.
\newblock {\em Adv. Math.}, 282:335--361, 2015.

\bibitem[Bor94]{BorceuxIII}
F.~Borceux.
\newblock {\em Handbook of categorical algebra. 3}, volume~52 of {\em
  Encyclopedia of Mathematics and its Applications}.
\newblock Cambridge University Press, Cambridge, 1994.
\newblock Categories of sheaves.

\bibitem[DHI04]{DHI}
D.~Dugger, S.~Hollander, and D.C. Isaksen.
\newblock Hypercovers and simplicial presheaves.
\newblock {\em Math. Proc. Cambridge Philos. Soc.}, 136(1):9--51, 2004.

\bibitem[FS02]{FriedlanderSuslin}
E.~M. Friedlander and A.~Suslin.
\newblock The spectral sequence relating algebraic {$K$}-theory to motivic
  cohomology.
\newblock {\em Ann. Sci. \'Ecole Norm. Sup. (4)}, 35(6):773--875, 2002.

\bibitem[Gra95]{Grayson}
D.~R. Grayson.
\newblock Weight filtrations via commuting automorphisms.
\newblock {\em $K$-Theory}, 9(2):139--172, 1995.

\bibitem[Gro64]{EGA41}
A.~Grothendieck.
\newblock \'{E}l\'ements de g\'eom\'etrie alg\'ebrique. {IV}. \'{E}tude locale
  des sch\'emas et des morphismes de sch\'emas. {I}.
\newblock {\em Inst. Hautes \'Etudes Sci. Publ. Math.}, (20):259, 1964.

\bibitem[Gro66]{EGA43}
A.~Grothendieck.
\newblock \'{E}l\'ements de g\'eom\'etrie alg\'ebrique. {IV}. \'{E}tude locale
  des sch\'emas et des morphismes de sch\'emas. {III}.
\newblock {\em Inst. Hautes \'Etudes Sci. Publ. Math.}, (28):255, 1966.

\bibitem[Gro67]{EGA44}
A.~Grothendieck.
\newblock \'{E}l\'ements de g\'eom\'etrie alg\'ebrique. {IV}. \'{E}tude locale
  des sch\'emas et des morphismes de sch\'emas. {IV}.
\newblock {\em Inst. Hautes \'Etudes Sci. Publ. Math.}, (32):361, 1967.

\bibitem[Hir03]{Hirschhorn}
P.~S. Hirschhorn.
\newblock {\em Model categories and their localizations}, volume~99 of {\em
  Mathematical Surveys and Monographs}.
\newblock American Mathematical Society, Providence, RI, 2003.

\bibitem[Hol08]{Hollander}
Sharon Hollander.
\newblock A homotopy theory for stacks.
\newblock {\em Israel J. Math.}, 163(1):93--124, 2008.

\bibitem[Hoy14]{Hoyois}
M.~Hoyois.
\newblock A quadratic refinement of the {G}rothendieck--{L}efschetz--{V}erdier
  formula.
\newblock {\em Alg. Geom. Top.}, 14(6):3603--3658, 2014.

\bibitem[Jar87]{Jardine}
J.~F. Jardine.
\newblock Simplicial presheaves.
\newblock {\em J. Pure Appl. Algebra}, 47(1):35--87, 1987.

\bibitem[Jar15]{JardineLHT}
J.~F. Jardine.
\newblock {\em Local homotopy theory}.
\newblock Springer Monographs in Mathematics. Springer, New York, 2015.

\bibitem[Lam06]{Lam}
T.~Y. Lam.
\newblock {\em Serre's problem on projective modules}.
\newblock Springer Monographs in Mathematics. Springer-Verlag, Berlin, 2006.

\bibitem[Lin82]{Lindel}
H.~Lindel.
\newblock On the {B}ass-{Q}uillen conjecture concerning projective modules over
  polynomial rings.
\newblock {\em Invent. Math.}, 65(2):319--323, 1981/82.

\bibitem[Lur09]{HTT}
J.~Lurie.
\newblock {\em Higher topos theory}, volume 170 of {\em Annals of Mathematics
  Studies}.
\newblock Princeton University Press, Princeton, NJ, 2009.

\bibitem[Lur11]{DAGXI}
J.~Lurie.
\newblock Derived algebraic geometry {XI}: Descent theorems.
\newblock 2011.
\newblock {\em Preprint,} available at
  \url{http://www.math.harvard.edu/~lurie/papers/DAG-XI.pdf}.

\bibitem[Lur16]{HA}
J.~Lurie.
\newblock {\em Higher {A}lgebra}.
\newblock 2016.
\newblock {\em Preprint,} available at
  \url{http://www.math.harvard.edu/~lurie/papers/HA.pdf}.

\bibitem[Mor12]{MField}
F.~Morel.
\newblock {\em {$\Bbb A^1$}-algebraic topology over a field}, volume 2052 of
  {\em Lecture Notes in Mathematics}.
\newblock Springer, Heidelberg, 2012.

\bibitem[Mos11]{Moser}
L.-F. Moser.
\newblock {${\mathbb A}^1$}-locality results for linear algebraic groups
  [draft].
\newblock {\em Preprint}, 2011.

\bibitem[MV99]{MV}
F.~Morel and V.~Voevodsky.
\newblock {${\mathbf A}^1$}-homotopy theory of schemes.
\newblock {\em Inst. Hautes \'Etudes Sci. Publ. Math.}, (90):45--143 (2001),
  1999.

\bibitem[Pop89]{Popescu}
D.~Popescu.
\newblock Polynomial rings and their projective modules.
\newblock {\em Nagoya Math. J.}, 113:121--128, 1989.

\bibitem[Rez]{Rezk}
C.~Rezk.
\newblock When are homotopy colimits compatible with homotopy pullback?
\newblock {\em Preprint}, available at
  \url{http://www.math.uiuc.edu/~rezk/i-hate-the-pi-star-kan-condition.pdf}.

\bibitem[{Rez}98]{RezkFib}
C.~{Rezk}.
\newblock {Fibrations and homotopy colimits of simplicial sheaves}.
\newblock 1998.
\newblock {\em Preprint} available at \url{http://arxiv.org/abs/math/9811038}.

\bibitem[Sch15]{Schlichting}
M.~Schlichting.
\newblock Euler class groups, and the homology of elementary and special linear
  groups.
\newblock {\em Preprint}, available at \url{http://arxiv.org/abs/1502.05424},
  2015.

\bibitem[{Sta}15]{stacks-project}
The {Stacks Project Authors}.
\newblock {\itshape Stacks Project}.
\newblock \url{http://stacks.math.columbia.edu}, 2015.

\bibitem[Ste51]{Steenrod}
N.~Steenrod.
\newblock {\em The {T}opology of {F}ibre {B}undles}.
\newblock Princeton Mathematical Series, vol. 14. Princeton University Press,
  Princeton, N. J., 1951.

\bibitem[Swa98]{SwanPopescu}
R.~G. Swan.
\newblock N\'eron-{P}opescu desingularization.
\newblock In {\em Algebra and geometry ({T}aipei, 1995)}, volume~2 of {\em
  Lect. Algebra Geom.}, pages 135--192. Int. Press, Cambridge, MA, 1998.

\bibitem[TT90]{Thomason}
R.~W. Thomason and T.~Trobaugh.
\newblock Higher algebraic {K}-theory of schemes and of derived categories.
\newblock In {\em The Grothendieck Festschrift III}, volume~88 of {\em Progress
  in Mathematics}, pages 247--435. Birkh{\"a}user, 1990.

\bibitem[Voe10]{VVcd}
V.~Voevodsky.
\newblock Homotopy theory of simplicial presheaves in completely decomposable
  topologies.
\newblock {\em J. Pure Appl. Alg.}, 214(8):1384--1398, 2010.

\bibitem[Vor81]{Vorst1}
T.~Vorst.
\newblock The general linear group of polynomial rings over regular rings.
\newblock {\em Comm. Algebra}, 9(5):499--509, 1981.

\bibitem[Vor83]{Vorst2}
T.~Vorst.
\newblock The {S}erre problem for discrete {H}odge algebras.
\newblock {\em Math. Z.}, 184(3):425--433, 1983.

\bibitem[Wei89]{Weibel}
C.~A. Weibel.
\newblock Homotopy algebraic {$K$}-theory.
\newblock In {\em Algebraic {$K$}-theory and algebraic number theory
  ({H}onolulu, {HI}, 1987)}, volume~83 of {\em Contemp. Math.}, pages 461--488.
  Amer. Math. Soc., Providence, RI, 1989.

\end{thebibliography}
\end{footnotesize}

\Addresses
\end{document}